\documentclass[a4paper]{amsart}
\usepackage[utf8]{inputenc}
\usepackage[T1]{fontenc}
\usepackage{amsfonts, amssymb, amscd, listings, xfrac, xifthen, wasysym, url, enumerate, mathrsfs, type1cm}
\usepackage[colorinlistoftodos]{todonotes}
\usepackage[polish,english]{babel}
\newtheorem{thm}{Theorem}[section]
\newtheorem{obs}[thm]{Observation}
\newtheorem{lem}[thm]{Lemma}

\newtheorem{prop}[thm]{Proposition}

\theoremstyle{remark}
\newtheorem{rem}{Remark}
\newcommand{\gD}{\mathfrak{D}}
\newcommand{\gp}{\mathfrak{p}}
\newcommand{\ogp}{\conj{\gp}}
\newcommand{\gq}{\mathfrak{q}}
\newcommand{\gr}{\mathfrak{r}}
\newcommand{\kk}{\mathbb{F}_q}

\newcommand{\FF}{\mathbb{F}}
\newcommand{\NN}{\mathbb{N}}
\newcommand{\QQ}{\mathbb{Q}}
\newcommand{\PP}{\mathbb{P}}
\newcommand{\ZZ}{\mathbb{Z}}
\newcommand{\CS}{\mathcal{S}}
\newcommand{\CO}{\mathcal{O}}
\newcommand{\SB}{\mathscr{B}}
\newcommand{\LL}{{F}}
\newcommand{\conj}[1]{\overline{#1}}

\newcommand{\even}{\equiv 0\pmod{2}}
\newcommand{\odd}{\equiv 1\pmod{2}}
\newcommand{\half}{\frac12}
\newcommand{\onto}{\twoheadrightarrow}
\newcommand{\iso}{\xrightarrow{\raisebox{-2bp}{\smash{\tiny$\,\sim\,$}}}}
\newcommand{\squares}[1]{#1^{\times2}}
\newcommand{\sing}[1]{E_{#1}}
\newcommand{\singK}{\sing{X}}
\newcommand{\Sing}[1]{\mathbf{E}_{#1}} 
\newcommand{\SingK}{\Sing{X}}
\newcommand{\units}[1]{#1^\times}
\newcommand{\dl}[1]{\Delta_{#1}}
\newcommand{\Dl}[1]{\mathbf{\Delta}_{#1}} 
\newcommand{\st}{\mid}
\DeclareMathOperator{\Norm}{Norm}
\newcommand{\N}[1][]{\ifthenelse{\equal{#1}{}}{\Norm_{K/\LL}}{\Norm_{#1}}}
\DeclareMathOperator{\norm}{norm}
\newcommand{\n}[1][]{\ifthenelse{\equal{#1}{}}{\norm_{K/\LL}}{\norm_{#1}}}
\newcommand{\rk}[1][]{\ifthenelse{\equal{#1}{}}{\rank_{2}}{\rank_{#1}}}
\newcommand{\Tt}[1][]{\ifthenelse{\equal{#1}{}}{T,t}{T_{#1},t_{#1}}}
\DeclareMathOperator{\Br}{Br}
\DeclareMathOperator{\Cl}{\Pic}
\DeclareMathOperator{\varCl}{Cl}
\DeclareMathOperator{\Div}{Div}
\DeclareMathOperator{\Pic}{Pic}
\newcommand{\PicK}{\Pic X}
\DeclareMathOperator{\Card}{card}
\newcommand{\card}[1]{\Card(#1)}
\DeclareMathOperator{\chr}{char}
\DeclareMathOperator{\dv}{div}

\DeclareMathOperator{\intcl}{int.cl.}

\DeclareMathOperator{\ord}{ord}
\DeclareMathOperator{\qf}{qf}
\DeclareMathOperator{\rank}{rk}
\newcommand{\sqgd}[1]{\sfrac{\units{#1}}{\squares{#1}}}
\newcommand{\form}[1]{\langle #1\rangle}
\newcommand{\aff}[1]{#1^{\text{aff}}}
\newcommand{\quat}[2]{\bigl(\frac{#1}{#2}\bigr)}
\newcommand{\wild}[1]{\mathcal{W}(#1)}
\newcommand{\term}[1]{\emph{#1}}
\author[A. Czoga{\l}a \and P. Koprowski \and B. Rothkegel]{Alfred Czoga\l a \and Przemys{\l}aw Koprowski \and Beata Rothkegel}
\address{Institute of Mathematics\\
University of Silesia\\ Bankowa 14\\ 40-007 Katowice,
Poland} \email{alfred.czogala@us.edu.pl}
\address{Institute of Mathematics\\
University of Silesia\\ Bankowa 14\\ 40-007 Katowice,
Poland} \email{przemyslaw.koprowski@us.edu.pl}
\address{Institute of Mathematics\\
University of Silesia\\ Bankowa 14\\ 40-007 Katowice,
Poland} \email{brothkegel@math.us.edu.pl}
\title{Wild and even points in global function fields}
\begin{document}
\begin{abstract}
We develop a criterion for a point of global function field to be a unique wild point of some self-equivalence of this field. We show that this happens if and only if the class of the point in the Picard group of the field is $2$-divisible. Moreover, given a finite set of points, whose classes are $2$-divisible in the Picard group, we show that there is always a self-equivalence of the field for which this is precisely the set of wild points. Unfortunately, for more than one point this condition is no longer a necessary one. 
\end{abstract}

\cleardoublepage

\maketitle
\section{Introduction and related works}
Hilbert-symbol equivalence (formerly known under the name \term{reciprocity equivalence}) appeared for the first time in early 90's in papers by J.~Carpenter, P.E.~Conner, R.~Litherland, R.~Perlis, K.~Szymiczek and the first author of this paper (see e.g. \cite{PSCL94}). It was originally introduced as a tool for investigating Witt equivalence of global fields (two fields are said to be \term{Witt equivalent}, when their Witt rings of similarity classes of non-degenerate quadratic forms are isomorphic---roughly speaking, Witt equivalent fields admit ``equivalent'' classes of orthogonal geometries). Nowadays, it is known that Witt equivalence of fields is closely related to \'etale cohomology. For fields of rational functions $K = k(X)$, the relevant groups are: $H^1(K,\sfrac\ZZ2)\cong \sqgd{K}$ the group of squares classes of $K$ and $H^2(K,\sfrac\ZZ2)\cong \Br_2(K)$, the group of 2-torsion elements in the Brauer group of $K$. When one passes to a finite extension of the field of rational functions, i.e. to the function field of an algebraic curve $X$, then the group $\sfrac{\PicK}{2\PicK}$ becomes relevant, too.

Recently, the theory of Hibert-symbol equivalence developed into a research subject by itself. It was generalized to higher-degree symbols (see e.g. \cite{CS97}, \cite{CS98}), to quaternion-symbol equivalence of real function fields (see e.g. \cite{Kop02a}), as well to a ring setting (see e.g. \cite{RC07}). One of the subjects considered in this theory is the problem of describing self-equivalences of a given field.

Let $K$ be a global field of characteristic $\neq 2$ and let $X$ denotes the set of all primes of $K$ (i.e. classes of non-trivial places on $K$). The \term{self-equivalence} of $K$ is a pair of maps $(\Tt)$, consisting of a bijection $T\colon X\iso X$ and an automorphism $t\colon \sqgd{K}\iso \sqgd{K}$ of the square-class group of $K$, satisfying the condition:
\[
(\lambda, \mu)_\gp = (t\lambda, t\mu)_{T\gp},\qquad\text{for all $\gp\in X$ and $\lambda,\mu\in \sqgd{K}$.}
\]
Here, $(\cdot, \cdot)_\gp$ denotes the Hilbert symbol $\sqgd{K_\gp}\times \sqgd{K_\gp}\to \{\pm 1\}$. Every self-equivalence of a global field induces an automorphism of its Witt ring. Given a self-equivalence of a global field $K$, a prime $\gp$ of $K$ is called \term{tame} if $\ord_\gp \lambda \equiv \ord_{T\gp} t\lambda\pmod{2}$ for all $\lambda\in K$. Otherwise $\gp$ is called \term{wild}. Few years ago, M.~Somodi gave a full characterization all finite sets of wild primes in $\QQ$ (see \cite{Som06}) and in $\QQ(i)$ (see \cite{Som08}). His results were recently generalized to a broad class of number fields by two of the authors of this article (for details see \cite{CR14}).

In this paper, we consider the same question for global function fields, i.e. algebraic function fields in one variable over finite fields. Hence from now on, $K$ is a global function field of characteristic $\neq 2$ and a (finite) field~$\kk$ is the full field of constants of~$K$. We may think of $K$ as of a field of rational functions on some smooth, irreducible complete curve $X$. The closed points of $X$ are identified with non-trivial places of $K$. We shall never explicitly refer to the generic point of $X$. Thus, in what follows, we use the word ``point'' meaning actually ``closed point''. We denote the set of closed points again by $X$.  We show (cf. Theorem~\ref{thm_even=wild}) that a point $\gp\in X$ is a unique wild point for some self-equivalence of $K$ if and only if its class in the Picard group of~$X$ is $2$-divisible (i.e. belongs to the subgroup $2\PicK$). On implication of this theorem, still holds even when we increase the number of points, this way we obtain a complete counterpart (Theorem~\ref{thm_main2}) for function fields of the results from \cite{Som06, Som08, CR14}. These two results establish a direct link between the property of being wild (for some-self equivalence) and $2$-divisibility in the Picard group of $K$. For this reason, we develop in Section~\ref{sec_2-div} some criteria for the class of a point $\gp\in X$ to be $2$-divisible in the group $\PicK$. In particular, we show (cf. Theorem~\ref{thm_2Pic<->NK}), that a point of a hyperelliptic curve (of an odd degree) is $2$-divisible in $\PicK$ (hence is a unique wild point of some self-equivalence) if and only if its norm over the rational function field is represented by the norm of the field extension $\sfrac{K}{\kk(x)}$. This in turn implies that for such curves, wild points always exist (see Proposition~\ref{prop_even_exist}).

In the paper we use the following notation (beside the one already explained above). Given a function field $K$ and a point $\gp\in X$, by $\CO_\gp$ we denote the associated valuation ring, by $K_\gp$ the completion of $K$ and by $K(\gp)$ the residue field. The degree $[K(\gp):\kk]$ of the residue field of~$\gp$ over the full field of constants is called the degree of~$\gp$ and denoted $\deg\gp$. Given a non-empty, open subset $Y\subseteq X$, we write $\CO_Y := \bigcap_{\gp\in Y} \CO_\gp$ and
\[
\sing{Y} := \bigl\{ \lambda\in \units{K}\st \forall_{\gp\in Y}\ord_\gp\lambda \equiv 0\pmod{2}\bigr\}
\]
This set is a union of cosets of $\squares{K}$ and we denote its image in the square-class group of $K$ by $\Sing{Y} := \sfrac{\sing{Y}}{\squares{{K}}}$. Further, when $Y$ is a proper subset, we consider the subset of $\sing{Y}$ consisting of all those functions that are local squares everywhere \emph{outside} $Y$, namely:
\[
\dl{Y} := \sing{Y}\cap \bigcap_{\gp\notin Y} \squares{K_\gp}= \singK\cap \bigcap_{\gp\notin Y} \squares{K_\gp}.
\]
This set again contains full square classes of $K$ and so we write $\Dl{Y} := \sfrac{\dl{Y}}{\squares{K}}$. In the special case, when $Y$ is of the form $X\setminus \{\gp\}$, i.e. is obtained from $X$ by excluding just a single point, we abbreviate the notation writing $\sing{\gp}$, $\Sing{\gp}$, $\dl{\gp}$ and $\Dl{\gp}$ for $\sing{X\setminus\{\gp\}}$, $\Sing{X\setminus\{\gp\}}$, $\dl{X\setminus\{\gp\}}$ and $\Dl{X\setminus\{\gp\}}$, respectively. 

The square-class group $\sqgd{\kk}$ has order two. We write $\zeta\in \kk\subset K$ for a fixed generator of this group, with the convention that $\zeta = -1$, whenever $-1$ is not a square in $K$ (i.e. when $\card{\kk}\equiv 3\pmod{4}$). Abusing the notation slightly, we tend to use the same symbols $\lambda, \mu,\dotsc$ to denote elements of the field as well as their classes in the square-class group of this field. Likewise, the fraktur letters $\gp, \gq,\dotsc$ denote, depending on the context, either points of $K$ or their classes in the Picard  group $\PicK$ or $\Pic \CO_Y$. Divisors, as well as their classes in the Picard group, are always written additively.

\section{Preliminaries}
Recall that if $K_\gp$ is a local field, then the square class group of $K_\gp$ consists of four elements: $1, u_\gp, \pi_\gp$ and $u_\gp\pi_\gp$, where $\pi_\gp$ is the class of a uniformizer and $u_\gp$ is the class of a unit which is not a square (see. e.g. \cite[Theorem~Vi.2.2]{Lam05}). We call $u_{\gp}$ the \term{$\gp$-primary unit}. If $(\Tt)$ is a self-equivalence of $K$, then $t$ factors over all the local square-class groups by \cite[Lemma~4]{PSCL94}. In particular it maps $1\in \sqgd{K_\gp}$ to $1\in \sqgd{K_{T\gp}}$. If it also maps $u_\gp$ to $u_{T\gp}$, then it is necessarily tame by the pigeonhole principle. Thus we proved:

\begin{obs}\label{obs_ord(tu)->wild}
The self-equivalence $(\Tt)$ is wild at a point $\gp\in X$ if and only if $\ord_{T\gp} tu_\gp\equiv 1\pmod{2}$.
\end{obs}

The primary unit $u_{\gp}$ may also be characterized using Hilbert symbols as follows:
\[
(u_{\gp},\lambda)_{\gp}=(-1)^{\ord_\gp\lambda}\qquad\text{for every }\lambda\in \units{K_{\gp}}.
\]
The Hilbert symbol $(\cdot, \cdot)_\gp$ can be viewed as a non-degenerate $\FF_2$-inner product on $\sqgd{K_{\gp}}$, provided the additive group $\FF_2$ is identified with the multiplicative group $\{\pm
1\}$. The following observation is now immediate:

\begin{obs}\label{obs_taup}
Let $\gp,\gq\in X$ be two points of $K$ such that $-1\in K_\gp^2\cap K_\gq^2$, then the isomorphism $\tau\colon \sqgd{K_\gp}\to \sqgd{K_\gq}$ defined by the formula:
\[
\tau(u_\gp)=u_\gq\pi_\gq,\qquad \tau(\pi_\gp)=\pi_\gq
\]
is an isometry of the inner product spaces $\bigl(\sqgd{K_\gp},(\cdot,\cdot)_\gp\bigr)$ and $\bigl(\sqgd{K_\gq}, (\cdot,\cdot)_\gq\bigr)$.
\end{obs}

Below we gather some results concerning $2$-ranks of the class groups: either the Picard group $\PicK$ of the complete curve $X$ or the Picard group $\Pic \CO_Y$ for some fixed open subset $\emptyset \neq Y\subsetneq X$. Recall that the latter group can be identified with the ideal class group $\varCl \CO_Y$ of the coordinate ring $\CO_Y$ of $Y$, as $\CO_Y$ is a Dedekind domain. We begin with a proposition that is not new. The first assertion was proved in \cite[p.~607]{Czo01} and the second in \cite[Lemma~2.1]{Czo01}. The third assertion is a simple consequence of the previous two. We state the result explicitly only to simplify further references.

\begin{prop}\label{prop_old_result}
Let $\emptyset \neq Y\subsetneq X$ be a proper open subset of $X$, then
\begin{enumerate}
\item\label{it_rk2E_Y} $\rk \Sing{Y} = \rk \Cl\CO_Y + \card{X\setminus Y}$;
\item\label{it_rk2Delta} $\rk \Dl{Y} = \rk \Cl \CO_Y$;
\item\label{it_rk2E_Y_var} $\rk \sfrac{\Sing{Y}}{\Dl{Y}} = \card{X\setminus Y}$;
\end{enumerate}
\end{prop}

An identity similar to \eqref{it_rk2E_Y} above can be also proved for a complete curve.

\begin{lem}\label{lem_rk2E_K}
$\rk \SingK = 1 + \rk \Pic^0 X$.
\end{lem}

\begin{proof}
Let $H$ be the subgroup of $\Pic^0 X$ consisting of elements of order $2$. The map
\[
\singK \to H, \quad \lambda\mapsto \half\dv_K\lambda = \sum_{\gp\in X}\half\ord_\gp\lambda\cdot \gp
\]
is a surjective homomorphism with the kernel $\units{\kk}\cdot \squares{K}$. Thus, $\rk \sfrac{\singK}{\units{\kk}\squares{K}}= \rk \Pic^0 X$. The groups $\sfrac{\units{\kk}\squares{K}}{\squares{K}}$ and $\sqgd{\kk}$ are isomorphic and the $2$-rank of $\sqgd{\kk}$ equals one. This proves the lemma.
\end{proof}

Now, we consider a case, when we have two open subsets $Z\subset Y\subset X$.

\begin{lem}\label{lem_rk2EZ_vs_rk2EY}
If $\emptyset \neq Z\subset Y\subsetneq X$ are two proper open subsets of $X$, then
\begin{enumerate}
\item $\rk \Cl\CO_Z = \rk \Cl\CO_Y - \rk\bigl\langle\{ \gp + 2\Cl\CO_Y\st \gp\in Y\setminus Z\}\bigr\rangle$;
\item $\rk \Sing{Z} = \rk \Cl\CO_Y - \rk\bigl\langle\{ \gp + 2\Cl\CO_Y\st \gp\in Y\setminus Z\}\bigr\rangle + \card{X\setminus Z}$.
\end{enumerate}
\end{lem}

\begin{proof}
Since $Z\subset Y$, hence $\CO_Z\supset \CO_Y$ and by functoriality, there is a natural morphism 
$\Cl\CO_Y\to \Cl\CO_Z$. It is clearly an epimorphism, since a class of a divisor $\sum_{\gp\in Z} n_\gp \gp$ is the image of the class of any divisor of a form $\bigl(\sum_{\gp\in Z}n_\gp\gp + \sum_{\gq\in Y\setminus Z} n_\gq\gq\bigr)$. This epimorphism induces an epimorphism of the quotient groups $\sfrac{\Cl\CO_Y}{2\Cl\CO_Y}\onto \sfrac{\Cl\CO_Z}{2\Cl\CO_Z}$, whose kernel is generated by the set $\{\gp + 2\Cl\CO_Y\st \gp\in Y\setminus Z\}$. It proves the first assertion of the lemma and the second follows immediately from the first one and Proposition~\ref{prop_old_result}.
\end{proof}

It is natural to compare the $2$-rank of $\Pic^0 X$ with the $2$-rank of the class group $\Cl\CO_Y$ of a proper open subset $Y\subsetneq X$. Below we formulate two relevant results for the case when $Y= X\setminus \{\gp\}$.

\begin{lem}\label{lem_EK=Ep}
Let $\zeta\in \kk$ be a fixed generator of the square class group $\sqgd{\kk}$ of the full field of constants of $K$. If $\gp\in X$ is a point of an odd degree, then
\begin{enumerate}
\item\label{it_E_K=E_gp} $\SingK = \Sing{\gp} = \langle\zeta\rangle\oplus\Dl{\gp}$;
\item $\rk \Pic^0K = \rk \Cl\CO_\gp$.
\end{enumerate}
\end{lem}

\begin{proof}
Take an arbitrary element $\lambda\in \sing{\gp}$. The degree of the principal divisor $\dv_K\lambda$ is $0$. Thus, we have
\[
\ord_\gp\lambda\cdot \deg \gp = - \sum_{\gq\neq\gp} \ord_\gq\lambda \cdot \deg \gq.
\]
Now, $\ord_\gq\lambda$ is even for every $\gq\neq\gp$, since $\lambda \in \sing{\gp}$. On the other hand $\deg\gp$ is odd by the assumption. It follows that $\ord_\gp\lambda$ is even, too. Hence $\lambda\in \singK$ and so we proved that $\Sing{\gp}\subseteq \SingK$. The other inclusion is trivial and the equality $\Sing{\gp} = \langle\zeta\rangle\oplus\Dl{\gp}$ follows from Proposition~\ref{prop_old_result} and the fact that $\zeta$ is not a local square at a given point if and only if this point has an odd degree. This proves~\eqref{it_E_K=E_gp}. Now, the other assertion follows immediately from Lemma~\ref{lem_rk2E_K} and Proposition~\ref{prop_old_result}.\eqref{it_rk2E_Y}.
\end{proof}

\begin{prop}\label{prop_rk2ClO_gp}
If $\gp\in X$ is any point, then
\[
\rk \Cl\CO_\gp = 
\begin{cases}
\rk \Pic^0 X,&\text{if $\gp\notin 2\PicK$}\\
1+\rk \Pic^0 X,&\text{if $\gp\in 2\PicK$.}
\end{cases}
\]
\end{prop}

The proof of this proposition will be postponed till the next section.

\section{2-divisibility of classes of prime divisors}\label{sec_2-div}
This section is fully devoted to the following problem: let $\gp\in X$ be a point, when does the class of $\gp$ in $\PicK$ is divisible by $2$ (i.e. lies in $2\PicK$)? Points having this property will be subsequently called \term{$2$-divisible} or shorty, albeit less formally, \term{even}. The results of this section, not only have direct applications in the rest of this paper, but (at least some of them) are somehow interesting by themselves. Let us begin with the following basic observation.

\begin{obs}\label{obs_even_pt->even_deg}
If $\gp\in X$ is an even point, then its degree $\deg\gp$ is an even integer.
\end{obs}

The assertion of this observation follows immediately from the fact (see e.g. \cite[Corollary~VII.7.10]{Lor96}), that the epimorphism $\deg: \Div K\onto \ZZ$ factors over the subgroup of principal divisors, inducing a well defined group epimorphism $\deg: \PicK\onto \ZZ$. It is well known (see e.g. \cite[Proposition~VII.7.12]{Lor96}), that for a field of rational functions this map is actually an isomorphism. Hence, in such a field, even points are precisely the points of even degrees. Of course, it is not so in general. For example, if $K$ is the function field of an elliptic curve over $\FF_3$ given in Weierstrass normal form by the polynomial $y^2-x^3+x$, then there are exactly $6$ points of degree $2$ and $12$ points of degree $4$ in $K$ but none (!) of them is $2$-divisible in $\PicK$ (verified using computer algebra system\footnote{The source codes for Magma of all the counter examples are available from the second author's web page at \url{http://z2.math.us.edu.pl/perry/papers}.} Magma \cite{BCP97}). Thus, we are forced to search for some other criteria of $2$-divisibility.

\begin{prop}\label{prop_even<->ord=1}
A point $\gp\in X$ is $2$-divisible in $\PicK$ if and only if there exists an element $\lambda \in \sing\gp$ such that $\ord_\gp \lambda\equiv 1\pmod{2}$.
\end{prop}

\begin{proof}
Assume that $\gp$ is an even point, this means that
\[
\gp +\dv_K\lambda = \sum_{\gq\in X} 2n_\gq\cdot \gq
\]
for some $n_\gq\in \ZZ$ almost all equal $0$ and some element $\lambda\in K$. It is clear that $\lambda$ satisfies the assertion.

Conversely, assume the existence of $\lambda\in \sing\gp$ of an odd order at $\gp$, say $\ord_\gp\lambda = 2k+1$. Write the divisor of $\lambda$ 
\[
\dv_K\lambda = (2k+1)\gp + \sum_{\substack{\gq\in X\\ \gq\neq \gp}} 2n_\gq \gq,
\]
for some $k\in \ZZ$ and $n_\gq\in \ZZ$ almost all equal $0$. Therefore the following equality holds in the Picard group of $K$
\[
\gp = \dv_K \lambda - 2\Bigl( k\gp + \sum_{\substack{\gq\in X\\ \gq\neq \gp}} n_\gq \gq\Bigr).
\]
In particular $\gp\in 2\PicK$, as claimed
\end{proof}

We will need the following, rather basic, fact from group theory, which we believe is well know to the experts but we are not aware of any convenient reference.

\begin{lem}\label{lem_2ranks}
Let $G$ be a finite abelian group. If~$H$ is a subgroup of~$G$, then
\[
\rk \sfrac{G}{H}\geq \rk G - \rk H.
\]
\end{lem}

\begin{proof}
The $2$-rank of a finite abelian group~$A$ is just the dimension of the $\FF_2$-vector space $A\otimes_{\ZZ} \FF_2$. Take a short exact sequence
\[
0\to H\to G \to \sfrac{G}{H}\to 0
\]
and tensor it with~$\FF_2$. We obtain the following exact sequence of $\FF_2$-vector spaces:
\[
H\otimes_{\ZZ}\FF_2\to G\otimes_{\ZZ}\FF_2\to \sfrac{G}{H}\otimes_{\ZZ}\FF_2\to 0.
\]
Let~$I$ be the image of the first homomorphism in the above sequence. Clearly $\dim_{\FF_2} (H\otimes_{\ZZ}\FF_2)\geq \dim_{\FF_2} I$ and we have
\[
\dim_{\FF_2} \bigl(G\otimes_{\ZZ} \FF_2\bigr) - \dim_{\FF_2} (H\otimes_{\ZZ}\FF_2)\leq 
\dim_{\FF_2} \bigl(G\otimes_{\ZZ} \FF_2\bigr) - \dim_{\FF_2} I = 
\dim_{\FF_2} (\sfrac{G}{H}\otimes_{\ZZ}\FF_2).
\]
This proves the lemma.
\end{proof}

We are now ready to prove Proposition~\ref{prop_rk2ClO_gp}.

\begin{proof}[Proof of Proposition~\ref{prop_rk2ClO_gp}]
Let $d:= \deg\gp$ be the degree of $\gp$. It follows from \cite[Proposition~14.1]{Ros02} that the following sequence is exact
\[
0\to \Pic^0 X\to \Cl\CO_\gp\to \ZZ_d\to 0.
\]
Therefore $\sfrac{\Cl\CO_\gp}{\Pic^0 X}$ is isomorphic to~$\ZZ_d$ and so their $2$-ranks are equal. Lemma~\ref{lem_2ranks} asserts that
\[
1\geq \rk\ZZ_d\geq \rk \Cl\CO_\gp - \rk\Pic^0 X.
\]
Consequently
\begin{equation}\label{eq_rk_ineq}
\rk \Cl\CO_\gp \leq 1 + \rk \Pic^0 X.
\end{equation}
Lemma~\ref{lem_rk2E_K} asserts that $\rk \SingK = 1 + \rk \Pic^0 X$, while Proposition~\ref{prop_old_result} states that $\rk \Cl\CO_\gp = \rk \Sing{\gp}-1$. Clearly $\SingK\subseteq \Sing{\gp}$. If $\gp\notin 2\PicK$, then $\SingK = \Sing{\gp}$ by Proposition~\ref{prop_even<->ord=1}, hence
\[
\rk \Cl\CO_\gp = \rk \Pic^0 X.
\]
On the other hand, if $\gp\in 2\PicK$, then $\SingK\subsetneq \Sing{\gp}$, again by Proposition~\ref{prop_even<->ord=1}. Thus
\[
\rk \Cl\CO_\gp > \rk \Pic^0 X
\]
and the assertion follows from Eq.~\eqref{eq_rk_ineq}.
\end{proof}

One immediate consequence of Proposition~\ref{prop_rk2ClO_gp} is the following criterion for $2$-divisibility.

\begin{prop}\label{prop_even<=>Dl=EE}
Let $\gp\in X$ be any point, then $\gp$ is $2$-divisible in $\PicK$ if and only if every function having even order everywhere on $X$ is a local square at $\gp$ \textup(i.e. if $\SingK = \Dl{\gp}$\textup).
\end{prop}

\begin{proof}
Think of $\Dl{\gp}$ as of a subspace of a $\FF_2$-linear space $\SingK$. Lemma~\ref{lem_rk2E_K} asserts that $\rk\SingK = 1 +\rk\Pic^0 X$, while $\rk\Dl{\gp} = \rk \Cl\CO_\gp$ by Proposition~\ref{prop_old_result}. Now, it follows from Proposition~\ref{prop_rk2ClO_gp} that $\rk\Cl\CO_\gp = 1+\rk \Pic^0 X= \rk \SingK$ if and only if $\gp\in 2\PicK$. Consequently, $\dim_{\FF_2} \Dl{\gp} = \dim_{\FF_2}\SingK$, and so $\Dl{\gp}$ is the full space $\SingK$, if and only if $\gp$ is even.
\end{proof}

So far we have been considering $2$-divisibility in the Picard group of the complete curve. The next proposition deals with $2$-divisibility in $\Pic\CO_Y$ (or equivalently in $\varCl\CO_Y$), that is over some proper open subset $Y$ of $X$.

\begin{prop}\label{prop_Delta<Kp^2}
Let $\emptyset \neq Y\subsetneq X$ be a proper open subset and $\gp\in Y$. Then $\gp$ is $2$-divisible in $\Cl\CO_Y$ if and only if $\dl{Y}\subset \squares{K_\gp}$.
\end{prop}

\begin{proof}
By the assumption there exists an element $\lambda\in \units{K}$ such that $\dv_{\CO_Y} \lambda = \gp + 2\gD$ for some $\CO_Y$-divisor $\gD\in \Div\CO_Y$. Fix an element $\mu\in \dl{Y}$. Then, for every $\gq\in X\setminus Y$, the element $\mu$ is a local square at $\gq$, hence the quaternion algebra $\quat{\lambda,\mu}{K_\gq}$ splits. On the other hand, when $\gq\in Y\setminus \{\gp\}$, then both $\mu$ and $\lambda$ are $\gq$-adic units modulo $\squares{K_\gq}$ and so again $\quat{\lambda,\mu}{K_\gq}$ splits. Consequently, all quaternion algebras $\quat{\lambda,\mu}{K_\gq}$ split for $\gq\in X$, except possibly $\gp$. It follows from Hilbert's reciprocity formula, that in such a case also $\quat{\lambda,\mu}{K_\gp}$ splits. But $\mu$ is arbitrary, which implies that $\lambda$ must be a local square at $\gp$.

Conversely, let $Z = Y\setminus \{\gp\}$. Since $\mu\in \squares{K_\gp}$ for every $\mu\in \dl{Y}$ by the assumption, hence $\Dl{Y} = \Dl{Z}$ and it follows from Proposition~\ref{prop_old_result}.\eqref{it_rk2Delta} that
\[
\rk \Cl\CO_Y=\rk \Cl\CO_Z.
\]
Consequently $\gp\in 2\Cl\CO_Y$, by the means of Lemma~\ref{lem_rk2EZ_vs_rk2EY}.
\end{proof}

Finally, we present a proposition connecting $2$-divisibility in the Picard group of a complete curve with $2$-divisibility over its open subset.

\begin{prop}\label{prop_2Pic<->2Cl}
Let $\gp$, $\gq$ be two points of $X$ with $\deg \gp$ even and $\deg \gq$ odd, then
\[
\gp \in 2\PicK\iff \gp \in 2\Cl\CO_{X\setminus\{\gq\}}.
\]
\end{prop}

\begin{proof}
Let $Y := X\setminus \{\gq\}$. If $\gp$ is $2$-divisible in $\PicK$, then $p = \dv_K\lambda + 2\gD$ for some $\lambda \in K$ and $\gD\in \Div K$. Drop any occurrences of $\gq$ in $\gD$ and the principal divisor $\dv_K\lambda$, to get $\CO_Y$-divisors $\gD'$ and $\dv_{\CO_Y}\lambda$. Therefore, over $\CO_Y$, we have
\[
\gp = \dv_{\CO_Y} \lambda + 2\gD'\in \Div \CO_Y
\]
and so $\gp\in 2\Cl\CO_Y$.

Conversely, assume that $\gp\in 2\Cl\CO_Y$, this means that there are: an element $\lambda\in K$ and $\CO_Y$-divisor $\gD\in \Div\CO_Y$ such that
\[
\dv_{\CO_Y}\lambda = p +2\gD\in \Div \CO_Y.
\]
Passing from $Y$ to the complete curve $X$, write
\[
\dv_K\lambda = \gp + 2\gD + \ord_\gq\lambda\cdot \gq.
\]
Compute the degrees of both sides to get
\[
0 = \deg \gp + 2\deg \gD + \ord_\gq \lambda\cdot \deg \gq.
\]
We assumed that $\deg \gq$ is odd, while $\deg \gp$ is even, hence $\ord_\gq\lambda$ must be even, too. Say $\ord_\gq\lambda = 2k$ for some $k\in \ZZ$. Thus, $\dv_K\lambda = \gp + 2(\gD+ k\gq)$, which means that $\gp$ is even, as desired.
\end{proof}

All the above results are of rather general nature and are valid for \emph{any} global function field. It should not come as a big surprise, that, if we concentrate on function fields of a special type, more can be proved. Recall that a smooth curve $X$, whose affine part $\aff{X}$ is defined by the polynomial $y^2- f(x)$, is called \term{hyperelliptic} when $\deg f\geq 4$, \term{elliptic} when $\deg f= 3$ and \term{conic} when $\deg f\leq 2$. In what follows, we will deal with elliptic and hyperelliptic curves in a uniform fashion, therefore, stretching the term a little, we shall call all curves of this form ``hyperelliptic'', treating elliptic curves as special case of hyperelliptic ones. We warn the reader, however, that this is not a standard terminology.

Let $\sfrac{K}{\LL}$ be an extension of function fields and $\pi: X\onto Y$ be the corresponding morphism of their associated (smooth) curves. Recall (cf. \cite[Ch.~VII,~\S7]{Lor96}) that a \term{norm} is a function $\N: \Div K\to \Div\LL$ given by the formula
\begin{equation}\label{eq_Norm}
\N\Bigl( \sum_i a_i \gp_i \Bigr) := \sum_i a_i f\bigl(\sfrac{\gp_i}{\pi(\gp_i)}\bigr) \pi(\gp_i),
\end{equation}
where $f\bigl(\sfrac{\gp}{\pi(\gp)}\bigr)$ is the inertia degree of $\gp$ over $\pi(\gp)$. If $\aff{Y}$ is an affine part of $Y$, $\CO_\LL = \kk[\aff{Y}]$ is the ring of functions regular on $\aff{Y}$ and $\CO_K = \intcl_K \CO_\LL$ is the integral closure of $\CO_\LL$ in $K$, then $\N\!\!\bigm|_{\Div\CO_K}$ restricted to $\Div \CO_K$ is a morphism $\N: \Div \CO_K\to \Div \CO_\LL$. If additionally $\LL = \kk(x)$ is a field of rational functions, then to every point  $\gp$ of $Y= \PP^1\kk$ one may unambiguously assign either a monic polynomial $p\in \kk[x]$ with a single zero at $\gp$ and no other zeros or a function $\sfrac1x$, when $\gp$ is the point at infinity. This constitutes a morphism $\Div\LL\to \units{\LL}$ from the group of divisors to the multiplicative group of the field~$\LL$. Composing it over $\N$, we arrive at the map $\n: \Div K\to \units{\LL}$, which (harmlessly abusing the notation) we shall again call a \term{norm}. In what follows, we shall prefer $\n$ over $\N$ since the former allows us to compare the norm of a divisor with values of the standard norm of the field extension $\n:\units{K}\to \units{\LL}$. 

\begin{thm}\label{thm_2Pic<->NK}
If $K$ is a function field of a smooth hyperelliptic curve $X$ of an odd degree and $\gp \in X$ is a point of an even degree. Then $\gp$ is $2$-divisible in $\PicK$ if and only if $\n\gp$ is representable by the $\n:\units{K}\to \units{\LL}$, where~$\LL$ is a field of rational functions. In other words
\[
\gp \in 2\PicK\iff \exists_{\lambda\in K} \n\gp = \n\lambda.
\]
\end{thm}

The proof of this theorem will be divided into three lemmas. In Lemmas~\ref{lem_BS}--\ref{lem_2Pic->NK}, $K = \qf\bigl( \sfrac{\kk[x,y]}{y^2-f(x)} \bigr)$ is always a function field of a hyperelliptic curve $X$ with its affine part defined by the polynomial $y^2- f(x)$, further $\LL = \kk(x)$ is a field of rational functions in $x$ and $\CO_K = \intcl \kk[x]$. We denote by $\conj{\vphantom{m}\ }:K\to K$ the unique non-trivial $\LL$-automorphism of $K$. The ring $\CO_K$ is a Dedekind domain, hence its Picard group can be identified with its ideal class group $\varCl\CO_K$.

The first lemma is basically a recap of \cite[Theorem~III.8.7]{BS66}. Unfortunately in \cite{BS66} it is proved only for number fields, hence for the sake of completeness we feel obliged to explicitly state and prove its function field counterpart.

\begin{lem}\label{lem_BS}
If the $\n\gD$ of a divisor $\gD\in \Div \CO_K$ equals $1$, then the class of $\gD$ lies in $2\Cl\CO_K$.
\end{lem}

\begin{proof}
We closely follow the lines of the proof of~\cite[Theorem~III.8.7]{BS66}. Write the divisor $\gD$ in the form
\[
\gD = \sum_{i=1}^m \bigl( a_i\gp_i + b_i\ogp_i\bigr) + \sum_{j=1}^n c_j\gq_j,
\]
where the points $\gq_j = \conj\gq_j$ are fixed under the action of $\conj{\vphantom{h}\ }$ and $\gp_i\neq \ogp_i$ are not. Then, $\n\gp_i = \n\ogp_i = p_i$ and $\n\gq_j = q_j^{f_j}$ for some monic polynomials $p_i, q_j\in \kk[x]$, $f_j\in \{1,2\}$, $i\leq m$, $j\leq n$. Therefore
\[
1 = \n\gD = \prod_{i=1}^m p_i^{a_i+b_i} \cdot \prod_{j=1}^n q_j^{c_j}.
\]
Now, all the polynomials are irreducible and pairwise distinct and $\kk[x]$ is a UFD, hence all the exponents must vanish. In particular $c_j = 0$ for every $j$ and $a_i = -b_i$ for every $i$. Consequently 
\[
\gD = \sum_{i=1}^m a_i\bigl( \gp_i-\ogp_i\bigr),
\]
but $\gp_i +\ogp_i= \dv_{\CO_K} p$, hence $\gp_i = - \ogp_i$ in the class group $\Cl\CO_K$. All in all, we write the class of $\gD$ as
\[
\sum_{i=1}^m 2 a_i\gp_i \in 2\Cl\CO_K.\qedhere
\]
\end{proof}

We are now in a position to prove the first implication of Theorem~\ref{thm_2Pic<->NK}.

\begin{lem}\label{lem_NK->2Pic}
If $\deg\gp\in 2\ZZ$ and $\n\gp\in \n\units{K}$, then $\gp$ is even.
\end{lem}

\begin{proof}
By the assumption of the theorem, the degree of $X$ is odd and it follows from \cite[Lemma~V.10.15]{Lor96} that $X$ has a unique point at infinity (denote it $\infty_K$) and this point is ramified. In particular, $\deg\infty_K= 1\notin 2\ZZ$ and so $\gp$ and $\infty_K$ are distinct. If the inertia degree of $\gp$ (in the field extension $\sfrac{K}{\LL}$) equals $2$, then $\n\gp = p^2$, for some monic polynomial $p\in \kk[x]$. This means that $\dv_Kp = \gp - 2\infty_K$. Therefore $\gp = \dv_Kp + 2\infty_K\in 2\PicK$. 

From now on, we assume that $\gp\neq \infty_K$ and the inertia degree of $\gp$ equals $1$. Hence, $\n\gp = p$  and by the assumption there exists such an element $\lambda\in K$, that $p = \n\lambda = \lambda\conj\lambda$. Take a divisor $\gD := \gp-\dv_{\CO_K}\lambda\in \Div\CO_K$. Clearly
\[
\n\gD = \frac{\n\gp}{\n\lambda} = 1
\]
and so the previous lemma asserts that $\gD\in 2\Cl\CO_K$. Since $\infty_K$ is the unique point at infinity and $\deg\infty_K= 1$, therefore \cite[Proposition~VIII.9.2]{Lor96} implies that the group $\Cl\CO_K$ is isomorphic to $\Pic^0K$. Hence, passing with $\gD$ to $\PicK$, we have $\gp - \dv_K\lambda+ 2k\infty_K \in 2\PicK$ for some $k\in \ZZ$. In particular $\gp \in 2\PicK$, as desired.
\end{proof}

We are now ready to prove the opposite implication of Theorem~\ref{thm_2Pic<->NK}.

\begin{lem}\label{lem_2Pic->NK}
The norm $\n\gp$ of every even point lies in $\n\units{K}$.
\end{lem}

\begin{proof}
Take a point $\gp\in X$ and assume that it is $2$-divisible in $\PicK$. Thus, there is a divisor $\gD\in \Div K$ and an element $\lambda\in K$ such that
\[
\gp = 2\gD + \dv_K\lambda.
\]
Compute the norms of both sides to get
\[
\n\gp = \n(2\gD + \dv_K\lambda) = (\n\gD)^2 \cdot \n\lambda.
\]
If $\lambda = a+by$ for some $a, b\in \LL$, then $\n\lambda = a^2 - b^2f$, therefore
\[
\n\gp = (ac)^2 - (bc)^2 f,
\]
where $c= \n\gD\in \LL$. In particular $\n\gp\in \n\units{K}$.
\end{proof}

The proof of Theorem~\ref{thm_2Pic<->NK} is now complete.

\begin{rem}\label{rem_odd_degree}
One should note that the condition $\deg f\notin 2\ZZ$ occurs only in the proof of Lemma~\ref{lem_NK->2Pic}. Therefore, the implication 
\[
\gp\in 2\PicK\Longrightarrow \n\gp\in \n\units{K}
\]
of the theorem holds even without this assumption. Nevertheless, for the other implication this condition is indispensable. Indeed, take 
\[
K = \qf\bigl( \sfrac{\FF_5[x,y]}{y^2-x^4+x+1} \bigr). 
\]
Using computer algebra system Magma one checks that there are a total of $8$ points of $K$ of degree $2$ that are not $2$-divisible in $\PicK$, but their norms lie in $\n\units{K}$.
\end{rem}

\begin{rem}
The assumption, that $\deg\gp$ is even, is also essential. Take a field 
\[
K = \qf\bigl( \sfrac{\FF_{13}[x,y]}{y^2 + 12x^3 + x^2 + 3x + 10} \bigr).
\]
As it was mentioned in the proof of Lemma~\ref{lem_NK->2Pic}, the field $K$ has the unique point at infinity $\infty_K$ and $\deg\infty_K = 1$. On the other hand, $\n\infty_K = \sfrac{1}{x}\in \n\units{K}$. Again this example was checked using Magma.
\end{rem}

The criterion in the above theorem, let us show that even points do exist.

\begin{prop}\label{prop_even_exist}
Let $K$ be a function field of a \textup(smooth\textup) hyperelliptic curve given by a polynomial $y^2-f(x)$. If $f\in \kk[x]$ is monic of an odd degree, then there are infinitely many points of $K$ that are $2$-divisible in $\PicK$.
\end{prop}

\begin{proof}
As observed in the proof of Lemma~\ref{lem_NK->2Pic}, $K$ has the unique point at infinity (denote it $\infty_K$). This point is ramified and the Picard group $\Cl\CO_K$ of $\CO_K = \intcl_K \kk[x]$ is isomorphic to $\Pic^0 X$. Let $f = f_1\dotsm f_n$ be the decomposition of $f$ into irreducible monic factors. Fix a non-zero integer $M\in \NN$ and take an irreducible polynomial $q\in \kk[x]$ of an even degree strictly greater than $M$ and prime to $\chr \kk$. Take an extension $\kk(\alpha_0)$ of~$\kk$, where $\alpha_0$ is a root of $q$. Clearly, $\kk(\alpha_0)\neq \kk$ since the degree of $q_0$ is even and greater than $M\neq 0$. Denote 
\[
\lambda_{1} := f_1(\alpha_0), \dotsc, \lambda_{n} := f_n(\alpha_0)
\]
and take a field $\kk(\beta) := \kk\bigl( \alpha_0, \sqrt{\lambda_{1}}, \dotsc, \sqrt{\lambda_{n}}\bigr)$. Further, let $p\in \kk[x]$ be the minimal polynomial of~$\beta$. Take $\gp\in X$ to be a point of $K$ dominating $p$. Clearly the degree of $\gp$ is even and we have\begin{equation}\label{eq_fi_squares_mod_p}
\Bigl( \frac{f_1}{p} \Bigr) = \dotsb = \Bigl( \frac{f_n}{p} \Bigr) = 1.
\end{equation}
If the inertia degree of $\gp$ equals $2$, then $\gp = \dv_{\CO_K} p$ in $\Div \CO_K$, hence $\gp = 0$ in $\Cl\CO_K\cong \Pic^0 X$. It follows that the class of $\gp$ in the Picard group $\PicK\cong \Pic^0 X\oplus \ZZ$ can be written as $(0, \deg \gp)$ and so clearly belong to $2\PicK$. Thus, assume that the inertia degree $f(\sfrac{\gp}{p})$ of $\gp$ is one.

We claim that $\n\gp\in\n K$ or in other words, that $p= \n\gp$ is represented over~$\LL= \kk(x)$ by the quadratic form $\form{1,-f}$. This is equivalent to saying that the form $\varphi:= \form{1,-f,-p}$ is isotropic over $\kk(x)$. By the local-global principle, it suffices to show that the form is locally isotropic in every completion of $\kk(x)$.

First, take the completion at infinity $\LL_\infty$. By the assumption, $-\ord_\infty f= \deg f\notin 2\ZZ$, while $-\ord_\infty p= \deg p\in 2\ZZ$. Decompose the form $\varphi\otimes \LL_\infty$ into the sum $\form{1,-p}\otimes \LL_\infty$ with coefficients of even order and $\form{-f}\otimes \LL_\infty$ of odd order. A well known consequence of Springer's theorem (see e.g. \cite[Proposition~VI.1.9]{Lam05}) asserts that $\varphi\otimes \LL_\infty$ is isotropic if and only if the residue form of $\form{1,-p}$ is isotropic. But the latter is just  $\form{1,-1}$, hence trivially isotropic, since $p$ is monic.

Take now a completion $\LL_s$ of $\LL$ at the place associated to some irreducible polynomial $s$ different from $p$ and not dividing $f$. Using \cite[Proposition~VI.1.9]{Lam05}, we see that $\varphi\otimes \LL_s$ is again isotropic, because its residue form has dimension three (over a finite field) and therefore is isotropic.

Next, consider the completion $\LL_p$ of $\LL$ at the place associated to $p$. We know, that all $f_i$'s are squares modulo $p$ and so is $f$ itself. Consequently $\form{1,-f}\otimes \LL_p$ is isotropic, hence $\varphi\otimes \LL_p$ is isotropic, too. Finally, take the $f_i$-adic completion $\LL_{f_i}$ for some monic irreducible factor $f_i$ of $f$. We have $(\frac{f_i}{p})= 1$ by~\eqref{eq_fi_squares_mod_p} and Dedekind's quadratic reciprocity law says that
\[
\Bigl(\frac{p}{f_i}\Bigr)\cdot \Bigl(\frac{f_i}{p}\Bigr) = (-1)^{\textstyle\sfrac{(\card{\kk}-1)(\deg f_i\cdot\deg p)}{2}},
\]
but $\deg p$ is even and it follows that $(\frac{p}{f_i})= 1$, too. Thus, $\varphi\otimes \LL_{f_i}$ is again isotropic. All in all, $\varphi$ is isotropic over $\LL$, which proves our claim. Theorem~\ref{thm_2Pic<->NK} asserts now that $\gp$ is even. It is immediate that taking subsequently $M := \deg p$ and repeating the above construction, we ultimately produce an infinite sequence of $2$-divisible points in $K$.
\end{proof}

%
%

\section{Main results}
In this section, we prove our two main results, namely: Theorem~\ref{thm_even=wild} showing that a point is even if and only it is a unique wild point for some self-equivalence and its partial generalization Theorem~\ref{thm_main2}. First, however, we need the following lemma, generalizing Proposition~\ref{prop_Delta<Kp^2}.

\begin{lem}\label{lem_Delta_basis}
Let $\emptyset\neq Y\subsetneq X$ be a proper open subset and $\gp_1, \dotsc, \gp_n\in Y$. Then $\gp_1, \dotsc, \gp_n$ are linearly independent \textup(over $\FF_2$\textup) in $\sfrac{\Cl\CO_Y}{2\Cl\CO_Y}$ if and only if there are $\lambda_1, \dotsc, \lambda_n\in \dl{Y}$ linearly independent in $\Dl{Y}$ and such that for every $1\leq i\leq n$
\[
\lambda_i\notin K_{\gp_i}^2\quad\text{and}\quad \lambda_i\in \bigcap_{j\neq i} K_{\gp_j}^2.
\]
\end{lem}

\begin{proof}
We proceed by an induction on the number of points. For $n=1$ the assertion follows from Proposition~\ref{prop_Delta<Kp^2}. Suppose that $n>1$ and the assertion holds true for $n-1$. Classes of $\gp_1, \dotsc, \gp_n$ are linearly independent in $\sfrac{\Cl\CO_Y}{2\Cl\CO_Y}$ and so, in particular, $\gp_1$ is not be $2$-divisible in $\Cl\CO_Y$. Proposition~\ref{prop_Delta<Kp^2} asserts that there exists $\mu\in \dl{Y}$ such that $\mu\notin K_{\gp_1}^2$. Take a subset $Z:= Y\setminus \{\gp_1\}$ of $Y$. By the means of Lemma~\ref{lem_rk2EZ_vs_rk2EY}, we have $\rk \Cl\CO_Z = \rk \Cl\CO_Y-1$. Clearly, $\Dl{Z}\subset\Dl{Y}$ with $\mu\in \Dl{Y}\setminus \Dl{Z}$. Moreover, $\gp_2, \dotsc, \gp_n$ remain linearly independent in $\sfrac{\Cl\CO_Z}{2\Cl\CO_Z}$.

It follows from the inductive hypothesis, that there are elements $\lambda_2, \dotsc, \lambda_n\in \dl{Z}$ linearly independent in $\Dl{Z}$ and such that for every $2\leq i\leq n$
\[
\lambda_i\notin K_{\gp_i}^2\quad\text{and}\quad \lambda_i\in \bigcap_{\substack{j\neq i\\j\geq 2}} K_{\gp_j}^2.
\]
By the very definition of $\dl{Z}$, all $\lambda_i$'s for $i\geq 2$ lie in $K_{\gp_1}^2$. Let
\[
\lambda_1 := \mu \cdot \prod_{i>1} \lambda_i^{\varepsilon_i},\qquad\text{where}\quad
\varepsilon_i = 
\begin{cases}
0, &\text{if $\mu\in K_{\gp_i}^2$}\\
1, &\text{if $\mu\notin K_{\gp_i}^2$.}
\end{cases}
\]
It is now immediate, that $\lambda_1\in \bigcap_{j\neq 1} K_{\gp_j}^2$ while at the same time $\lambda_1\notin K_{\gp_1}^2$. This proves one implication. The other one follows from \cite[Lemma~2.1]{Czo01}.
\end{proof}

\begin{lem}
Let $\gp\in 2\PicK$ be an even point, then for any other even point $\gq\in 2\PicK$, the set $\Sing{\gp}\setminus \SingK$ is fully contained in a square-class of the completion~$K_\gq$.
\end{lem}

\begin{proof}
Since $\Dl{\gp}= \SingK$ by Proposition~\ref{prop_even<=>Dl=EE}, thus $\SingK$ is a subgroup of $\Sing{\gp}$ of index $(\Sing{\gp}: \SingK)=2$ by Proposition~\ref{prop_old_result}. Take any $\lambda, \mu\in \Sing{\gp}\setminus \SingK$, then $\lambda\cdot \SingK= \mu\cdot \SingK$ and so $\lambda\cdot \mu\in \SingK= \Dl{\gq}\subset \squares{K_\gq}$.
\end{proof}

We define a relation on the set of $2$-divisible points, saying that $\gp\in 2\PicK$ is related to $\gq\in 2\PicK$, when $\Sing{\gp}\setminus \SingK\subset \squares{K_\gq}$. We write $\gp\smile \gq$, when $\gp$ is related to $\gq$. Unfortunately this relation---although symmetric---is neither reflexive nor transitive (see Remark~\ref{rem_nontransitive_smile}). 

\begin{lem}\label{lem_symmetric_smile}
The relation $\smile$ is symmetric.
\end{lem}

\begin{proof}
Take $\lambda\in \Sing{\gp}\setminus \SingK$ and $\mu\in \Sing{\gq}\setminus \SingK$. Assume that $\gp$ is related to $\gq$, so that $\lambda\in \squares{K_\gq}$. Take any point $\gr$ distinct from both $\gp$ and $\gq$, then a local quaternion algebra $\quat{\lambda,\mu}{K_\gr}$ splits, since $\ord_\gr\lambda\equiv \ord_\gr\mu\equiv 0\pmod{2}$. Next, also $\quat{\lambda,\mu}{K_\gq}$ splits, because $\lambda$ is a square in $K_\gq$. It follows from Hilbert's reciprocity law that $\quat{\lambda,\mu}{K_\gp}$ splits, as well. But $\ord_\gp\lambda\equiv 1\pmod{2}$, hence $\mu$ must be a local square at $\gp$. Consequently $\Sing{\gq}\setminus \SingK$ is contained in $\squares{K_\gp}$ and so $\gq$ is related to $\gp$.
\end{proof}

\begin{rem}\label{rem_nontransitive_smile}
While it is obvious (and harmless) that $\smile$ is not reflexive, but it is less obvious that in general it is not transitive. Take the function field of an elliptic curve $X$ over $\FF_3$ given by the equation $y^2 = x^3 + x - 1$. Consider the points $\gp, \gq, \gr\in X$, where $\gp$ is the common zero of $x$ and $x^3+x$; $\gq$ is the common zero of $x^4 + x^2 + 2x + 1$ and $y + x^2 + 2x$ and finally $\gr$ is the common zero of $x^4 + x^2 + 2x + 1$ and $y + 2x^2 + x$. Then, using computer algebra system Magma one can check that,  $\gp\smile \gq$ and $\gp \smile \gr$, nevertheless the points $\gq$ and $\gr$ are not related.
\end{rem}

Let us now recall the notion of a \term{small equivalence}. Let $\emptyset\neq \CS\subset X$ be a finite (hence closed) subset of $X$. We say that $\CS$ is \term{sufficiently large} if $\rk\Cl\CO_{X\setminus \CS} = 0$. If $\CS\subset X$ be a sufficiently large set of points of $K$, then a triple $\bigl(T_{\CS},t_{\CS},(t_{\gp}\st \gp\in \CS)\bigr)$ is called (cf. \cite[\S6]{PSCL94}) a \term{small $\CS$-equivalence} of the field $K$ if
\begin{enumerate}\renewcommand{\theenumi}{SE\arabic{enumi}}
\item $T_{\CS}\colon \CS\to X$ is injective,
\item $t_{\CS}\colon \Sing{X\setminus\CS}\to \Sing{X\setminus T\CS}$ is a group isomorphism,
\item\label{it_hilbert_pres} for every $\gp\in\CS$ the map $t_{\gp}\colon \sqgd{K_{\gp}}\to \sqgd{K_{T\gp}}$ is an isomorphism of local square class groups preserving Hilbert symbols, in the sense that
\[
(x,y)_{\gp}=(t_{\gp}x,t_{\gp}y)_{T\gp},\qquad\text{for all }x,y\in \sqgd{K_{\gp}};
\]
\item the following diagram commutes
\begin{equation}\label{eq_diagram}
\begin{CD}
\Sing{X\setminus\CS} @>i_{\CS}>> \prod\limits_{\gp\in\CS}\sqgd{K_{\gp}}\\
  @VVt_{\CS}V @VV\prod_{\gp\in\CS}t_{\gp}V\\
\Sing{X\setminus T\CS} @>i_{T\CS}>> \prod\limits_{\gp\in\CS}\sqgd{K_{T\gp}}
\end{CD}
\end{equation}
where the maps $i_{\CS}=\prod_{\gp\in\CS} i_{\gp}$ and $i_{T\CS}=\prod_{\gq\in T\CS} i_{\gq}$ are the diagonal homomorphisms, $i_{\gp}\colon \Sing{X\setminus\CS}\to \sqgd{K_{\gp}}$, $i_{\gq}\colon \Sing{X\setminus T\CS}\to \sqgd{K_{\gq}}$.
\end{enumerate}
We say that the local isomorphism $t_{\gp}\colon \sqgd{K_{\gp}}\to\sqgd{K_{T\gp}}$ is \term{tame}, when
\[
\ord_\gp \lambda\equiv \ord_{T\gp}t_{\gp}\lambda \pmod{2}\qquad\text{for every }\lambda\in \sqgd{K_\gp}.
\]
The following result follows from \cite[Theorem~2 and Lemma~4]{PSCL94}:

\begin{thm}\label{thm_PSCL}
Every small $\CS$-equivalence $\bigl(T_{\CS},t_{\CS},(t_{\gp}\st \gp\in \CS)\bigr)$ of the field $K$ can be extended to a self-equivalence $(\Tt)$ of $K$ tame on $X\setminus\CS$. Moreover, the self-equivalence $(\Tt)$ is tame at $\gp\in \CS$ if and only if the local isomorphism $t_{\gp}$ is tame.
\end{thm}

\begin{rem}
In the case considered in this paper (that is over global function fields) any local square-class group $\sqgd{K_\gp}$ consists of just four elements $\{1, u_\gp, \pi_\gp, u_\gp\pi_\gp\}$, with $\ord_\gp u_\gp\even$ and $\ord_\gp \pi_\gp\odd$. For two square classes $\lambda, \mu\in \sqgd{K_\gp}$, $\lambda, \mu\neq 1$, the Hilbert symbol can be computed with the formula
\[
(\lambda, \mu)_\gp = 1\iff \lambda = \mu.
\]
Therefore, \emph{every} bijection of the local square-class groups mapping squares to squares is an isomorphism and preserves the Hilbert symbols. Consequently, the condition \eqref{it_hilbert_pres} is always satisfied for this type of fields. 
\end{rem}

\begin{prop}\label{prop_even->wild}
Let $K$ be a global function field and $X$ an associated smooth curve. Let $\gp, \gp_1, \dotsc, \gp_l$ be $2$-divisible points such that $\gp_i\smile \gp_j$ for every $i\neq j$ Then, there is a self-equivalence $(\Tt)$ of $K$ such that:
\begin{itemize}
\item $\gp$ is the unique wild point of $(\Tt)$, i.e. $\wild{\Tt} = \{\gp\}$;
\item $T$ preserves the selected points in the sense that 
\[ 
T\gp = \gp\quad\text{and}\quad T\gp_i= \gp_i\text{ for }i=1, \dotsc, l;
\]
\item for every $\gp_i\smile\gp$, the isomorphism $t$ restricted to the local square-class group $\sqgd{K_{\gp_i}}$ is the identity;
\item for every $\gp_i\not\smile\gp$, the isomorphism $t$ restricted to the local square-class group $\sqgd{K_{\gp_i}}$ is a transposition of the square-classes of odd orders.
\end{itemize}
\end{prop}

\begin{proof}
Take an open subset $Y:= X\setminus \{\gp; \gp_1, \dotsc, \gp_l\}$ of $X$ and let $m := \rk \Cl\CO_Y$. Observe that
\begin{multline*}
\rk \Dl{Y} = \rk \Cl \CO_Y = \\
= \rk \Cl\CO_\gp - \rk\langle \gp_1 + 2\Cl\CO_\gp, \dotsc, \gp_l + 2\Cl\CO_\gp\rangle =\\
= \rk \Cl\CO_\gp = \rk \Dl{\gp},
\end{multline*}
where the first and the last equalities follows from Proposition~\ref{prop_old_result}, the second follows from Lemma~\ref{lem_rk2EZ_vs_rk2EY}, while the third  one is due to the fact that every $\gp_i$ is $2$-divisible in $\PicK$ so consequently also in $\Cl \CO_\gp$. Therefore, the $\FF_2$-linear spaces $\Dl{\gp}$ and $\Dl{Y}$ are equal, but the former one is just $\SingK$ by Proposition~\ref{prop_even<=>Dl=EE}. All in all, $\Dl{Y} = \SingK$. Take a basis $\gq_1, \dotsc, \gq_m$ of $\sfrac{\Cl\CO_Y}{2\Cl\CO_Y}$. Lemma~\ref{lem_Delta_basis} asserts that there are elements $\mu_1, \dotsc, \mu_m\in \dl{Y}$ linearly independent in $\Dl{Y}$ and such that $\mu_i \in \squares{K_{\gq_j}}$ if and only if $i\neq j$. Clearly, they form a basis of $\Dl{Y} = \SingK$. Now, $\rk \sfrac{\Sing{\gp}}{\SingK} = 1$ by Proposition~\ref{prop_even<=>Dl=EE} and Proposition~\ref{prop_old_result}. Likewise, $\rk \sfrac{\Sing{\gp_i}}{\SingK} = 1$ for every $i= 1, \dotsc, l$. Therefore, there are square-classes 
\[
\lambda \in \Sing{\gp}\setminus \SingK,\quad
\lambda_1 \in \Sing{\gp_1}\setminus \SingK,\dotsc, 
\lambda_l \in \Sing{\gp_l}\setminus \SingK.
\]
By the assumption $\gp_i\smile \gp_j$ for every $1\leq i\neq j\leq l$, hence every $\lambda_i$ is a local square at every $\gp_j$ for $j\neq i$. Multiplying, if necessary by appropriate $\mu_j$'s we may assume without loss of generality, that $\lambda, \lambda_1, \dotsc, \lambda_l$ are local squares at $\gq_j$ for every $j = 1, \dotsc, m$.

Denote $\CS:= \{ \gp; \gp_1, \dotsc, \gp_l; \gq_1, \dotsc, \gq_m\}$ and let $Z := X\setminus \CS \subset Y$. It follows from Lemma~\ref{lem_rk2EZ_vs_rk2EY} that $\rk\Cl\CO_Z = 0$ and so $\CS$ is a sufficiently large set. We claim that the set $\SB := \{ \lambda; \lambda_1, \dotsc, \lambda_l; \mu_1, \dotsc, \mu_m\}$ forms a basis of the $\FF_2$-linear space $\Sing{Z}$. First, we show that it is linearly independent. Suppose, a contrario, that it is not. Thus
\[
\nu := \lambda^a \cdot \prod_{i=1}^l \lambda_i^{b_i}\cdot \prod_{j=1}^m \mu_j^{c_j}
\]
is a square in $K$, for some $a, b_1, \dotsc, b_l, c_1, \dotsc, c_m\in \FF_2$. This means that $0\equiv \ord_\gp \nu\equiv a\pmod{2}$, since all the other elements have even order at $\gp$, consequently $a = 0$. Similarly, for every $1\leq i\leq l$, $0\equiv \ord_{\gp_i}\nu\equiv b_i\pmod{2}$ so also $b_1= \dotsb= b_l = 0$. Finally, $c_1 = \dotsb = c_m = 0$, because $\mu_1, \dotsc, \mu_m$ are linearly independent in $\Dl{Y}$, a subspace of $\SingK$. Further, Proposition~\ref{prop_old_result} asserts that $\dim_{\FF_2}\Sing{Z} = \rk \Cl\CO_Z + \card\CS = \card\SB$, proving that $\SB$ is a basis of $\Sing{Z}$.

Observe that, if $\gp$ is related to \emph{every} point $\gp_i$, $i=1, \dotsc, l$, then a $\gp$-primary unit $u$ does not belong to $\Sing{Z}$. On the other hand, if $\gp\not\smile \gp_i$ for some $i\in \{1, \dotsc, l\}$, then the obtained above element $\lambda_i$ is a $\gp$-primary unit (and symmetrically $\lambda$ is a $\gp_i$-primary unit).

Construct a triple $\bigl( T_\CS, t_\CS, (t_\gr\st \gr\in \CS)\bigr)$ in the following way:
\begin{itemize}
\item let $T_\CS\colon \CS\to \CS$ be the identity;
\item define the automorphism $t_\CS\colon \Sing{Z}\to \Sing{Z}$ fixing its values on the basis $\SB$:
\begin{itemize}
\item $t_\CS(\lambda) := \lambda$;
\item $t_\CS(\lambda_i) := \begin{cases} \lambda_i,& \text{if $\gp\smile \gp_i$,}\\ \lambda\lambda_i,& \text{if $\gp\not\smile\gp_i$;}\end{cases}$
\item $t_\CS(\mu_j) := \mu_j$ for $j= 1,\dotsc, m$.
\end{itemize}
\item finally, the automorphisms of the local square-class groups are given as follows:
\begin{itemize}
\item $t_\gp$ is the transposition $(u,u\lambda)$ on $\sqgd{K_\gp} = \{1, u, \lambda, u\lambda\}$ (recall that $u = \lambda_i\pmod{\squares{K_\gp}}$ whenever $\gp\not\smile\gp_i$);
\item for a point $\gp_i$ related to $\gp$, take $t_{\gp_i}$ to be the identity on $\sqgd{K_{\gp_i}}$;
\item for a point $\gp_i$ not related to $\gp$, let $t_{\gp_i}$ be a ``tame transposition'' $(\lambda_i, \lambda\lambda_i)$ on the group $\sqgd{K_{\gp_i}}= \{1, \lambda, \lambda_i, \lambda\lambda_i\}$;
\item eventually, for the remaining points $\gq_1, \dotsc, \gq_m$, let $t_{\gq_j}$ be the identity on the corresponding square class group.
\end{itemize}
\end{itemize}
The commutativity of the diagram~\eqref{eq_diagram} is now immediate. It follows that the triple $\bigl( T_\CS, t_\CS, (t_\gr\st \gr\in \CS)\bigr)$ is a small equivalence and Theorem~\ref{thm_PSCL} asserts that it can be extended to a self-equivalence $(\Tt)$ of $K$ tame on $Z$. Since only $t_\gp$ is wild, $\gp$ is the unique wild point of $(\Tt)$.
\end{proof}

\begin{lem}
Let $K$ be a global function field and $X$ an associated smooth curve and $(\Tt)$ be a self-equivalence of $K$. If $(\Tt)$ has a unique wild point $\gp$, then $\gp\in 2\PicK$.
\end{lem}

\begin{proof}
By the assumption $\wild{\Tt} = \{\gp\}$. Denote $\gq := T\gp$. Suppose, a contrario, that $\gp$ is not $2$-divisible.  Thus, Proposition~\ref{prop_even<->ord=1} asserts that every element of $\sing{\gp}$ has an even order at $\gp$, in particular $\Sing{\gp} = \SingK$. Now, it follows from Proposition~\ref{prop_old_result}\eqref{it_rk2E_Y_var}, that there is an element $\lambda\in K$ such that $\SingK = \Sing{\gp} = \langle \lambda\rangle\oplus \Dl{\gp}$. Clearly, $\ord_\gp\lambda \equiv 0\pmod{2}$ and $\lambda$ is not a local square at $\gp$, that is $\lambda$ is a $\gp$-primary unit. 

Now, $\gp$ is a wild point of $(\Tt)$, hence $\ord_\gq t\lambda\equiv 1\pmod{2}$ by Observation~\ref{obs_ord(tu)->wild}. It follows from Proposition~\ref{prop_even<->ord=1} that $\gq$ is an even point of $K$. It is straightforward to show that $t\Sing{\gp}= \Sing{T\gp}= \Sing{\gq}$. In particular, the $2$-ranks must agree:
\[
\rk \Sing{\gp} = \rk \Sing{\gq}.
\]
Use Proposition~\ref{prop_old_result} to express these $2$-ranks as:
\[
\rk \Cl\CO_\gp + 1= \rk \Cl\CO_\gq + 1.
\]
Now, $\gq$ is $2$-divisible in $\PicK$, while $\gp$ is not. Proposition~\ref{prop_rk2ClO_gp} asserts that the left-hand side equals $\rk \Pic^0 X+1$, while the right-hand side is $\rk \Pic^0 X+2$. This is clearly a contradiction.
\end{proof}

Combining now Proposition~\ref{prop_even->wild} with the above lemma, we arrive at our first main result.

\begin{thm}\label{thm_even=wild}
Let $K$ be a global function field and $X$ an associated smooth curve. Given a point $\gp\in X$, the following two conditions are equivalent:
\begin{itemize}
\item $\gp$ is $2$-divisible in $\PicK$;
\item $\gp$ is the unique wild point of some self-equivalence of $K$.
\end{itemize}
\end{thm}

Looking at Proposition~\ref{prop_even->wild} it is obvious that if we have a set of even points and each of them is related to every other, then we can build a number of self-equivalences, each wild at precisely one of these points and preserving the rest. Then the wild set of the composition of all these self-equivalences consists of all our (related) even points. It turns out that this is still true, even when not all the points are related. Theorem~\ref{thm_main2} below, not only generalizes one implication of Theorem~\ref{thm_even=wild}, but also constitutes a direct counterpart of \cite[Theorem~1.1]{CR14} for the case of global function fields.

\begin{thm}\label{thm_main2}
Let $K$ be a global function field and $X$ be its associated smooth curve. Given a finitely many points $\gp_1, \dotsc, \gp_n\in X$, that are $2$-divisible in $\PicK$, there is a self-equivalence $(\Tt)$ of $K$ such that $\gp_1, \dotsc, \gp_n$ are precisely its wild points, i.e. $\wild{\Tt} = \{\gp_1, \dotsc, \gp_n\}$.
\end{thm}

\begin{proof}
We proceed by an induction on the number $n$ of the points. The case $n=1$ simply boils down to the assertion of Theorem~\ref{thm_even=wild}. Hence, suppose that the thesis holds for all sets of cardinality $n-1$ and consider a set of $n$ even points $\{\gp_1, \dotsc, \gp_n\}\subset X$. Since $\gp_1$ is even, thus Proposition~\ref{prop_even->wild} asserts that there exists a self-equivalence $(\Tt[1])$ of $K$ such that $\gp_1$ is the unique wild point of $(\Tt[1])$ and $T_1\gp_1 = \gp_1$. Denote the images of the remaining points by $\gq_2 := T_1\gp_2, \dotsc, \gq_n := T_1\gp_n$. We claim that the points $\gq_2, \dotsc, \gq_n$ are all $2$-divisible in $\PicK$.

In order to prove the claim, observe first, that since $\gp_1$ is even, thus $\Dl{\gp_1} = \SingK$ by Proposition~\ref{prop_even<=>Dl=EE}. Moreover $(\Tt[1])$ is tame on $X\setminus \{\gp_1\}$, therefore $t_1\Sing{\gp_1} = \Sing{T_1\gp_1}= \Sing{\gp_1}$. It follows that also $t_1\Dl{\gp_1} = t_1(\Sing{\gp_1}\cap \squares{K_{\gp_1}})= \Sing{\gp_1}\cap \squares{K_{\gp_1}} = \Dl{\gp_1}$, as every self-equivalence preserves local squares. Consequently we obtain:
\[
t\SingK = t_1\Dl{\gp_1} = \Dl{\gp_1} = \SingK.
\]
Take now any point $\gp_i$ with $i>1$ and write
\[
\SingK = t_1\SingK = t_1\Dl{\gp_i} = t_1\bigl( \SingK\cap \squares{K_{\gp_i}} \bigr) = t_1\SingK\cap t_1\squares{K_{\gp_i}} = \SingK\cap \squares{K_{\gq_i}} = \Dl{\gq_i}.
\]
It follows from Proposition~\ref{prop_even<=>Dl=EE} that $\gq_i\in 2\PicK$, as claimed.

By the inductive hypothesis, there exists a self-equivalence $(\Tt[2])$ of $K$ with the wild set $\wild{\Tt[2]} = \{ \gq_2, \dotsc, \gq_n\}$. The composition $(\Tt)= (T_2\circ T_1, t_2\circ t_1)$ is now the desired self-equivalence of $K$ with the wild set $\wild{\Tt} = \{\gp_1, \dotsc, \gp_n\}$.
\end{proof}

\begin{rem}
The above theorem generalizes only one of the implications of Theorem~\ref{thm_even=wild} to sets having more than one point. This is all we can do, since the opposite implication no longer holds for larger sets. The simplest counterexample, we are aware of, is probably the following one: let $K$ be the function field of an elliptic curve over $\FF_5$ given in Weierstrass normal form by the polynomial $y^2 + x^3 + x + 2$. Take two points: $\gp\sim (1,1)$ and $\gq\sim (1,4)$. Then neither of them is even, since both are rational. Nevertheless, there exists a self-equivalence of $K$ that is wild precisely at these two points. We will discuss the structure of bigger wild set in another paper.
\end{rem}

\subsubsection*{Acknowledgments}
We wish to thank the anonymous reviewer for providing a corrected proof of Lemma~\ref{lem_2ranks} and for useful comments that improved the overall exposition of the paper.

\bibliographystyle{alpha} 
\bibliography{main}
\end{document}